%% file: main.tex
\documentclass[a4paper,pdftex,12pt]{article}
\usepackage[utf8]{inputenc}
\usepackage{a4wide}
\usepackage{amsmath,amssymb,amsfonts,amscd,amsthm,amsxtra,stmaryrd,mathrsfs,nicefrac}
\usepackage{graphicx,xcolor,url,hyperref}
\usepackage{float,subfigure,caption,wrapfig}
  
\hypersetup{
     colorlinks   = true,
     citecolor    = blue
}
\usepackage{dsfont, booktabs, enumitem}

\usepackage{listings}
\usepackage{color}
\definecolor{deepblue}{rgb}{0,0,0.5}
\definecolor{deepred}{rgb}{0.6,0,0}
\definecolor{deeporange}{rgb}{1,0.15,0}
\definecolor{deepgreen}{rgb}{0,0.5,0}
\definecolor{deepgray}{rgb}{0.4,0.4,0.4}
\definecolor{gray}{rgb}{0.5,0.5,0.5}
\DeclareFixedFont{\ttm}{T1}{txtt}{m}{n}{12}

\newcommand\pythonstyle{\lstset{
	language=Python,
	basicstyle=\ttm,
	otherkeywords={=},
	morekeywords={[2]output},
	morekeywords={[3]cone_solver},
	keywordstyle=\color{deepred},
	keywordstyle={[2]\color{deeporange}},
	keywordstyle={[3]\color{deeporange}},
	emph={	print, 
			dict %
		},
	emphstyle=\color{deepblue},
	stringstyle=\color{deepgreen},
	literate=%
				*{0}{{\textcolor{deepblue}{0}}}{1}%
				{1}{{\textcolor{deepblue}{1}}}{1}%
				{2}{{\textcolor{deepblue}{2}}}{1}%
				{3}{{\textcolor{deepblue}{3}}}{1}%
				{4}{{\textcolor{deepblue}{4}}}{1}%
				{5}{{\textcolor{deepblue}{5}}}{1}%
				{6}{{\textcolor{deepblue}{6}}}{1}%
				{7}{{\textcolor{deepblue}{7}}}{1}%
				{8}{{\textcolor{deepblue}{8}}}{1}%
				{9}{{\textcolor{deepblue}{9}}}{1}%
				{.0}{{\textcolor{deepblue}{.0}}}{2}%
				{.1}{{\textcolor{deepblue}{.1}}}{2}%
				{.2}{{\textcolor{deepblue}{.2}}}{2}%
				{.3}{{\textcolor{deepblue}{.3}}}{2}%
				{.4}{{\textcolor{deepblue}{.4}}}{2}%
				{.5}{{\textcolor{deepblue}{.5}}}{2}%
				{.6}{{\textcolor{deepblue}{.6}}}{2}%
				{.7}{{\textcolor{deepblue}{.7}}}{2}%
				{.8}{{\textcolor{deepblue}{.8}}}{2}%
				{.9}{{\textcolor{deepblue}{.9}}}{2}%
				{Exception}{{\textcolor{deepblue}{Exception}}}{8}%
				{BROJA\_2PID}{{\textcolor{deepblue}{BROJA\_2PID}}}{8}%
				{**}{{\textcolor{deepred}{$^{**}$}}}{2}%
																	,
	frame=tb,                         
	showstringspaces=false,
	numbers=left,            
	numbersep=5pt, 
	numberstyle=\tiny\color{gray},
	commentstyle=\color{deepgray} 
}
}
\lstnewenvironment{python}[1][]
{
\pythonstyle
\lstset{#1}
}
{}

\newcommand\pythonexternal[2][]{{
\pythonstyle
\lstinputlisting[#1]{#2}}
}

\newcommand\pythoninline[1]{{\pythonstyle\lstinline!#1!}} 

\newcommand\shellstyle{\lstset{
	language=bash,
	basicstyle=\ttm,
	keywordstyle=\color{deepred},
	emphstyle=\color{deepblue},
	stringstyle=\color{deepgreen},
	frame=tb,                         
	showstringspaces=false,
	numbers=left,            
	numbersep=5pt, 
	numberstyle=\tiny\color{gray},
	commentstyle=\color{deepgray} 
}
}
\newcommand\shellinline[1]{{\shellstyle\lstinline!#1!}}

\input{defsN.amsart}
\input{defs.common}

\DeclareMathOperator{\MI}{MI}
\DeclareMathOperator{\UI}{UI}
\DeclareMathOperator{\CI}{CI}

\DeclareMathOperator{\SI}{SI}
\DeclareMathOperator{\CP}{CP}
\DeclareMathOperator{\Ima}{Im}

\newcommand{\y}{\texttt{y}} \newcommand{\z}{\texttt{z}}
\newcommand{\norm}[1]{{\lt\lVert{#1}\rt\rVert}}

\begin{document}
\hypersetup{linkcolor=blue}
\title{BROJA-2PID: A robust estimator for bivariate partial information decomposition}
\author{Abdullah Makkeh, Dirk Oliver Theis\thanks{Supported by the Estonian Research Council, ETAG (\textit{Eesti Teadusagentuur}), through PUT Exploratory Grant \#620, and by the European Regional Development Fund through the Estonian Center of Excellence in Computer Science, EXCS.}, and Raul Vicente \\[1ex]
  \small Institute of Computer Science {\tiny of the } University of Tartu\\
  \small \"Ulikooli 17, 51014 Tartu, Estonia\\
  \small \{\texttt{makkeh,dotheis,raul.vicente.zafra}\}\texttt{@ut.ee}}



\maketitle

\begin{abstract}
    Makkeh, Theis, and Vicente found in~\cite{makkeh2017bivariate} that Cone Programming model is the most robust to compute the Bertschinger et al.~partial information decompostion (BROJA PID) measure~\cite{bertschinger-rauh-olbrich-jost-ay:quantify:2014}. We developed a production-quality robust software that computes the BROJA PID measure based on the Cone Programming model. In this paper, we prove the important property of strong duality for the Cone Program and prove an equivalence between the Cone Program and the original Convex problem.  Then describe in detail our software and how to use it.\newline\indent
    \textbf{Keywords:} Bivariate Information Decomposition; Cone Programming
\end{abstract}
\section{Introduction}\input{intro}
\section{Cone Programming Model for Bivariate PID}\label{sec:expcone}\input{cone_prog}
\section{The BROJA\_2PID Estimator}\label{sec:imp}\input{implementation}
\section{Tests}\label{sec:case-study}\input{case_study}
\section{Outlook}\input{conclusion}

\end{document}

%% file: defsN.amsart.tex
\theoremstyle{plain}
\newtheorem{theorem}{Theorem}
\newtheorem*{theorem*}{Theorem}

\newtheorem{proposition}[theorem]{Proposition}
\newtheorem*{proposition*}{Proposition}

\newtheorem*{corollary*}{Corollary}

\newtheorem*{lemma*}{Lemma}

\newtheorem*{observation*}{Observation}

\newtheorem*{conjecture*}{Conjecture}

\newtheorem*{question*}{Question}

\newtheorem*{questions*}{Questions}

\newtheorem*{problem*}{Problem}

\newtheorem*{problems*}{Problems}

\newtheorem*{openproblem*}{Open Problem}

\theoremstyle{definition}
\newtheorem{definition}[theorem]{Definition}
\newtheorem*{definition*}{Definition}

\newtheorem*{example*}{Example}

\newtheorem*{exercise*}{Exercise}
\newtheorem*{notation*}{Notation}
\theoremstyle{remark}

\newtheorem*{remark*}{Remark}

\newtheorem*{remarks*}{Remarks}

\newtheorem*{claim*}{Claim}

\newcommand{\subclass}[1]{}

\newcommand{\enumTi}[1]{\renewcommand{\theenumi}{#1}}

\newcommand{\alphenumi}{\enumTi{\alph{enumi}}}

\newcommand{\romenumi}{\enumTi{\roman{enumi}}}


%% file: defs.common.tex
\newlength{\hspaceforlengthglumpf}

\renewcommand{\em}{\sl}

\newcommand{\lt}{\left}
\newcommand{\rt}{\right}

\newcommand{\abs}[1]{{\lt\lvert{#1}\rt\rvert}}

\newcommand{\nfrac}[2]{{\nicefrac{#1}{#2}}}

\newcommand{\NN}{\mathbb{N}}

\newcommand{\RR}{\mathbb{R}}

\DeclareMathOperator*{\Prb}{\mathbf{P}}

\newlength{\algotabbingwidth}


%% file: intro.tex
For random variables $X,Y,Z$ with finite range, consider the mutual information $\MI(X;Y,Z)$: the amount of information that the pair $(Y,Z)$ contain about~$X$. How can we quantify the contributions of $Y$ and~$Z$, respectively, to $\MI(X;Y,Z)$? This question is at the heart of \textit{(bivariate) partial information decomposition, PID~\cite{harder2013bivariate,williams2010nonnegative,griffith2014quantifying,bertschinger-rauh-olbrich-jost-ay:quantify:2014}.} Information theorists agree that there can be: information shared redundantly by $Y$, and~$Z$; information contained uniquely within~$Y$ but not within~$Z$; information contained uniquely within~$Z$ but not within~$Y$; and information that synergistically results from combining both $Y$ and~$Z$. The quantities are denoted by: $\SI(X;Y,Z)$; $\UI(X;Y\backslash Z)$, $\UI(X;Z\backslash Y)$; and $\CI(X;Y,Z)$.  All four of these quantities add up to $\MI(X;Y,Z)$; moreover, the quantity of total information that $Y$ has about $X$ is decomposed into the quantity of unique information that $Y$ has about $X$ and shared information that $Y$ shares with $Z$ about $X$, similarly, for quantity of total information that $Z$ has about $X$ and so, $\SI(X;Y,Z) + \UI(X;Y\backslash Z) = \MI(X;Y)$, and $\SI(X;Y,Z) + \UI(X;Z\backslash Y) = \MI(X;Z)$.  Hence, if the joint distribution of~$(X,Y,Z)$ is known, then there is (at most) one degree of freedom in defining a bivariate PID.  In other words, defining the value of one of the information quantities defines a bivariate PID.

Bertschinger et al.~\cite{bertschinger-rauh-olbrich-jost-ay:quantify:2014} have given a definition of a bivariate PID where the synergistic information is defined as follows:
\begin{equation}\label{eq:broja-short-def}
		\CI(X;Y,Z) := \max( \MI(X;Y,Z) - \MI(X';Y',Z') )
\end{equation}
where the maximum extends over all triples of random variables $(X',Y',Z')$ with the same 12,13-marginals as $(X,Y,Z)$, i.e., $P(X=x, Y=y) = P(X'=x,Y'=y)$ for all $x,y$ and $P(X=x, Z=z) = P(X'=x,Z'=z)$ for all $x,z$.  It can easily be verified that this amounts to maximizing a concave function over a compact, polyhedral set easily described by inequalities.  Hence, using standard theorems of convex optimization, BROJA's bivariate PID can be efficiently approximated to any given precision.

In practice, computing $\CI$ has turned out to be quite a bit more challenging, owed to the fact that the objective function is not smooth on the boundary of the feasible region, which results in numerical difficulties for the state-of-the-art interior point algorithms for solving convex optimization problems.  We refer to~\cite{makkeh2017bivariate} for a thorough discussion of this phenomenon.

Due to these challenges, and the need in the scientific computing community to have a reliable, easily usable software for computing the BROJA bivariate PID, we made available on GitHub a Python implementation of our best method for computing the BROJA bivariate PID (\href{https://github.com/Abzinger/BROJA_2PID/}{github.com/Abzinger/BROJA\_2PID/}).  The solver is based on a conic formulation of the problem and thus a Cone Program is used to compute the BROJA bivariate PID. This paper has two contributions. Firstly, we prove the important property of strong duality for the Cone Program and prove an equivalence between the Cone Program and the original Convex problem. Secondly, we describe in detail our software and how to use it.

This paper is organized as follows. In the remainder of this section, we define some notation we will use throughout, and review the Convex Program for computing the BROJA bivariate PID from~\cite{bertschinger-rauh-olbrich-jost-ay:quantify:2014}.  In the next section we review the math underlying our software to the point which is necessary to understand how it works and how it is used.  In Section~\ref{sec:imp}, we walk the reader through an example of how to use the software --- and then explain its inner workings and its use in detail. In Section~\ref{sec:case-study}, we present some computations on larger problem instances, and discuss how the method scales up.  We conclude the paper by discussing our plans for the future development of the code.

\subsection{Notation and background}
Denote by $\mathbf{X}$ the range of the random variable~$X$, by $\mathbf{Y}$ the range of~$Y$, and by $\mathbf{Z}$ the range of~$Z$.  We identify joint probability density functions with points in $\RR^{\mathbf W}$, e.g., the joint probability distribution of $(X,Y,Z)$ is a vector in $\RR^{\mathbf{X}\times\mathbf{Y}\times\mathbf{Z}}$.  (We measure information in nats, unless otherwise stated.) We use the following notational convention.
\begin{quote}
An asterisk stands for ``sum over everything that can be plugged in instead of the $*$''.  E.g., if $p,q\in\RR^{\mathbf{X}\times\mathbf{Y}\times\mathbf{Z}}$,
\begin{equation*}
	q_{x,y,*} = \sum\nolimits_{w\in Z} q_{x,y,w};\qquad p_{*,y,z} q_{*,y,z} = \left( \sum\nolimits_{u\in X} p_{u,y,z} \right) \left(\sum\nolimits_{u\in X} q_{u,y,z}\right)
\end{equation*}
We do not use the symbol~$*$ in any other context.
\end{quote}
We define the following notation for the marginal distributions of $(X,Y,Z)$: With~$p$ the joint probability density function of $(X,Y,Z)$: 
\begin{align*}
	b^\y_{x,y} &:= p_{x,y,*} = \Prb\bigl( X = x \land Y = y \bigr) && \text{for all $x\in \mathbf{X}$, $y\in \mathbf{Y}$}\\
	b^\z_{x,z} &:= p_{x,*,z} = \Prb\bigl( X = x \land Z = z \bigr) && \text{for all $x\in \mathbf{X}$, $y\in \mathbf{Y}$.}
\end{align*}

\noindent\medskip%
These notations allow us to write the Convex Program from~\cite{bertschinger-rauh-olbrich-jost-ay:quantify:2014} in a succinct way.
Unraveling the objective function of~\eqref{eq:broja-short-def}, we find that, given the marginal conditions, is equal, up to a constant not depending on $X',Y',Z'$ to the conditional entropy $H(X'\mid Y',Z')$. Replacing maximizing $H(\dots)$ by minimizing $-H(\dots)$, we find~\eqref{eq:broja-short-def} to be equivalent to the following Convex Program: 
\begin{equation}\label{eq:the-BROJA-CP}\tag{\mbox{CP}}
	\begin{aligned}
		\text{minimize }         &\sum_{x,y,z} q_{x,y,z}\ln\frac{q_{x,y,z}}{q_{*,y,z}}  &&\text{ over $q\in\RR^{\mathbf{X}\times \mathbf{Y}\times \mathbf{Z}}$}\\
		\text{subject to }\quad  & q_{x,y,*} = b^\y_{x,y}        &&\text{ for all $(x,y)\in \mathbf{X}\times \mathbf{Y}$}\\
		{}                       & q_{x,*,z} = b^\z_{x,z}        &&\text{ for all $(x,z)\in \mathbf{X}\times \mathbf{Z}$}\\
		{}                       & q_{x,y,z} \ge 0               &&\text{ for all $(x,y,z)\in \mathbf{X}\times \mathbf{Y}\times \mathbf{Z}$.}
	\end{aligned}
\end{equation}

%% file: cone_prog.tex
In~\cite{makkeh2017bivariate}, we introduce a model for computing the BROJA bivariate PID based on a so-called ``Cone Programming''.  Cone Programming is a far reaching generalization of Linear Programming: The usual inequality constraints which occur in Linear Programs can be replaced by so-called ``generalized inequalities'' --- see below for details. Similar to Linear Programs, dedicated software is available for Cone Programs, but each type of generalized inequalities (i.e., each cone) requires its own algorithms. The specific type of generalized inequalities needed for the computation of the BROJA bivariate PID requires solvers for the so-called ``Exponential Cone'', of which a few are available.

In the computational results of~\cite{makkeh2017bivariate}, we found that the Cone Programming approach (based on one of the available solvers) was, while not the fastest, the most robust of all methods for computing the BROJA bivariate PID which we tried (and we tried a lot).  This is why our software is based on the Exponential Cone Programming model.

In this section, we review the mathematical definitions to the point in which they are necessary to understand our model and the properties of the software based on it.

\subsection{Background on Cone Programming}
A nonempty \textit{closed convex cone} $\mathcal K\subseteq\RR^m$ is a closed set which is \textit{convex}, i.e., for any $x,y\in\mathcal K$ and $0\le\theta\le1$ we have 
	$$\theta x + (1-\theta)y \in\mathcal K,$$ 
and is a \textit{cone}, i.e., for any $x\in\mathcal K$ and $\theta\ge 0$ we have 
	$$\theta x\in\mathcal K.$$
E.g., $\RR_+^n$ is a closed convex cone. \textit{Cone Programming} is a far-reaching generalization of Linear Programming, which may contain so-called \textit{generalized inequalities}: For a fixed closed convex cone~$\mathcal K\subseteq\RR^m$, the generalized inequality ``$a \le_{\mathcal K} b$'' denotes $b-a \in \mathcal K$ for any $a,b\in\RR^m$. Recall the primal-dual pair of Linear Programming. The \textit{primal problem} is, 

\begin{equation}\label{eq:lp-primal}
 	\begin{aligned}
 		\text{minimize } 	& c^Tw\\
 		\text{subject to } 	& Aw = b \\
 							& Gw \le h
 	\end{aligned}
\end{equation}
over variable $w\in\RR^n$, where $A\in\RR^{m_1\times n}, G\in\RR^{m_2\times n}, c\in\RR^n,b\in\RR^{m_1},$ and $h\in\RR^{m_2}$. And its \textit{dual problem} is, 
\begin{equation}\label{eq:lp-dual}
 	 	\begin{aligned}
 	 		\text{maximize } 	& -b^T\eta - h^T\theta\\
 	 		\text{subject to } 	& -A^T\eta - G^T\theta = c\\
 	 		 	 				& \theta \ge 0.
 	 	\end{aligned}
\end{equation}
There are two properties that the pair~\eqref{eq:lp-primal} and \eqref{eq:lp-dual} may or may not have, namely, weak and strong duality. The following defines the duality properties.

\begin{definition}\label{def:duality}
Consider a primal-dual pair of the Linear Program~\eqref{eq:lp-primal}, \eqref{eq:lp-dual}. Then we define the following,
\begin{enumerate}
	\item A vector $w\in\RR^n$ (resp.~$(\eta,\theta)\in\RR^{m_1}\times\RR^{m_2}$) is said to be a \textit{feasible solution} of \eqref{eq:lp-primal} (resp.~\eqref{eq:lp-dual}) if $Aw=b$ and $Gw\le_{\mathcal K} h$ (resp.~$-A^T\eta - G^T\theta = c$ and $\theta \ge 0$), i.e., none of the constraints in~\eqref{eq:lp-primal} (resp.~\eqref{eq:lp-dual}) are violated by $w$ (resp.~$(\eta,\theta)$).
	\item We say that~\eqref{eq:lp-primal} and~\eqref{eq:lp-dual} \textit{satisfy weak duality} if for all~$w$ and all~$(\eta,\theta)$ feasible solutions of~\eqref{eq:lp-primal} and~\eqref{eq:lp-dual} respectively, 
	$$-b^T\eta -h^T\theta\le c^T w.$$
	\item If $w$ is a feasible solution of~\eqref{eq:primal} and $(\eta,\theta)$ is a feasible solution of~\eqref{eq:dual}, then the \textit{duality gap} $d$ is $$d:= c^Tw + b^T\eta +h^T\theta .$$
	\item We say that~\eqref{eq:lp-primal} and~\eqref{eq:lp-dual} \textit{satisfy strong duality} when the feasible solutions $w$ and $(\eta,\theta)$ are optimal in~\eqref{eq:lp-primal} and~\eqref{eq:lp-dual} respectively if and only if $d$ is zero.
\end{enumerate}
\end{definition} 

Weak duality always hold for a Linear Program, however, strong duality hold for a Linear Program whenever a feasible solution of~\eqref{eq:lp-primal} or~\eqref{eq:lp-dual} exists. These duality properties are used to certify the optimality of $w$ and $(\eta,\theta)$. The same concept of duality exists for Cone Programming, the \textit{primal cone problem} is
\begin{equation}\label{eq:primal}\tag{\mbox{P}}
 	\begin{aligned}
 		\text{minimize } 	& c^T w\\
 		\text{subject to } 	& A w = b \\
 					& G w \le_{\mathcal K} h,
 	\end{aligned}
\end{equation}
over variable $w\in\RR^n$, where $A\in\RR^{m_1\times n}, G\in\RR^{m_2\times n}, c\in\RR^n,b\in\RR^{m_1},$ and $h\in\RR^{m_2}$. The \textit{dual cone problem} is,
\begin{equation}\label{eq:dual}\tag{\mbox{D}}
 	 	\begin{aligned}
 	 		\text{maximize } 	& -b^T\eta - h^T\theta\\
 	 		\text{subject to } 	& -A^T\eta - G^T\theta = c\\
 	 					& \theta \ge_{\mathcal K^*} 0,
 	 	\end{aligned}
\end{equation}
where $\mathcal K^*:=\{u\in\RR^n\mid u^Tv \ge 0~\text{for all}~v\in\mathcal K\}$ is the \textit{dual cone} of $\mathcal K$. The entries of the vector $\eta\in\RR^{m_1}$ are called the \textit{dual variables for equality constraints},~$Aw=b$. Those of $\theta\in\RR^{m_2}$ are the \textit{dual variables for generalized inequalities},~$Gw \le_{\mathcal K} h$. The primal-dual pair of a Conic Optimization~\eqref{eq:primal} and~\eqref{eq:dual} satisfy weak and strong duality in the same manner as the Linear Programming pair. In what follows, we will define the interior point of a Cone Program and then state when weak and strong duality (see Definition~\ref{def:duality}) hold for the Conic Programming pair.

\begin{definition}\label{def:int-pt}
	Consider a primal-dual pair of the Conic Optimization~\eqref{eq:primal}, \eqref{eq:dual}. Then the primal problem~\eqref{eq:primal} has an interior point $\tilde{x}$ if,
		\begin{itemize}
			\item $\tilde{x}$ is a feasible solution of~\eqref{eq:primal}.
			\item There exists $\epsilon>0$ such that for any $y\in\RR^n$, we have $y\in\mathcal K$  whenever $\norm{h - G\tilde{x} - y}_2\le\epsilon.$
		\end{itemize}
\end{definition}
\begin{theorem}[Theorem 4.7.1~\cite{gartner2012approximation}]\label{thm:dual}
		Consider a primal-dual pair of the Conic Optimization~\eqref{eq:primal}, \eqref{eq:dual}. Let $w$ and $(\eta,\theta)$ be the feasible solutions of~\eqref{eq:primal} and~\eqref{eq:dual} respectively. Then,
		\begin{enumerate}
			\item Weak duality always hold for~\eqref{eq:primal} and~\eqref{eq:dual}.
			\item If $c^Tw$ is finite and~\eqref{eq:primal} has an interior point $\tilde{w}$, then strong duality holds for~\eqref{eq:primal} and~\eqref{eq:dual}.
		\end{enumerate}
\end{theorem}  
If the requirements of Theorem~\ref{thm:dual} are met for a conic optimization problem, then weak and strong duality can be used as guarantees that the given solution of a Cone Program is optimal. One of the closed convex cones which we will use throughout the paper is the \textit{exponential cone}, $\mathcal K_{\exp}$, defined in~\cite{chares2009cones} as
\begin{equation}\label{eq:expcone}
	\{(r,p,q)\in\RR^{3}\mid q > 0 \text{ and } q e^{r/q} \le p\}\cup\{(r,p,0)\in\RR^3\mid r\le 0 \text{ and } p\ge 0\},
\end{equation}
which is the closure of the set
\begin{equation}\label{eq:open-expcone}
	\{(r,p,q)\in\RR^{3}\mid q > 0 \text{ and } q e^{r/q} \le p\},
\end{equation}
and its dual cone, $\mathcal K^*_{\exp},$ is 
\begin{equation}\label{eq:dual-expcone}
	\{(u,v,w)\in\RR^{3}\mid u<0\text{ and } -u\cdot e^{w/u}\le e\cdot v\}\cup\{(0,v,w)\mid v\ge0\text{ and } w\ge0\},
\end{equation}
which is the closure of the set
\begin{equation}\label{eq:open-dual-expcone}
	\{(u,v,w)\in\RR^{3}\mid u<0\text{ and } -u\cdot e^{w/u}\le e\cdot v\}.
\end{equation}
When $\mathcal K = \mathcal K_{\exp}$ in~\eqref{eq:primal} then the Cone Program is referred to as ``Exponential Cone Program''. 

\begin{figure}[H]
	\captionsetup[subfigure]{justification=centering}
	\centering
	\hspace*{\fill}%
	\subfigure[$\mathcal K_{\exp}$ for $ -2\le r\le 0$ and \hspace*{0.5cm}  $0\le q,p \le 2.$]{\includegraphics[height=4.96cm]{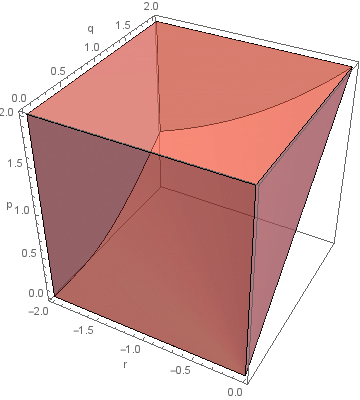}}\hfill
	\subfigure[$\mathcal{K^*_{\exp}}$ for $ -2\le u\le 0$ and \hspace*{0.5cm} $0\le w,v \le 2.$]{\includegraphics[height=4.96cm]{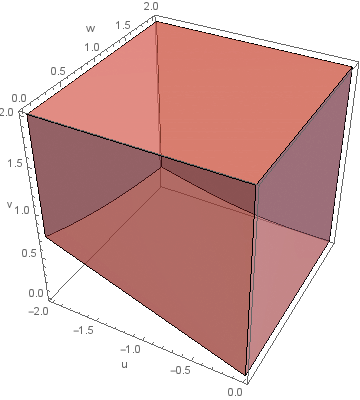}}
	\hspace*{\fill}%
	\caption{The $\mathcal K_{\exp}$ cone and its dual.}
\end{figure}
\subsection{The Exponential Cone Programming Model}
The Convex Program~\eqref{eq:the-BROJA-CP} which computes the bivariate partial information decomposition can be formulated as an Exponential Cone Program. Consider the following Exponential Cone Program where the variables are $r,p,q \in \RR^{\mathbf{X}\times\mathbf{Y}\times\mathbf{Z}}$.

\begin{equation}\label{eq:BROJA-ExpCone}\tag{\mbox{EXP}}
 \begin{aligned}
	  \text{minimize }       & -\sum_{x,y,z} r_{x,y,z}\\
	  \text{subject to}\quad & q_{x,y,*}             = b^\y_{x,y}  								&&\text{ for all $(x,y)\in \mathbf{X}\times\mathbf{Y}$}\\
	  {}                     & q_{x,*,z}             = b^\z_{x,z}  								&&\text{ for all $(x,z)\in \mathbf{X}\times\mathbf{Z}$}\\
	  {}                     & q_{*,y,z} - p_{x,y,z} =  0										&&\text{ for all $(x,y,z)\in \mathbf{X}\times \mathbf{Y}\times\mathbf{Z}$}\\
	  {}                     & (-r_{x,y,z},-p_{x,y,z},-q_{x,y,z}) \le_{\mathcal K_{\exp}} 0 	&&\text{ for all $(x,y,z)\in \mathbf{X}\times\mathbf{Y}\times\mathbf{Z}$.}\\
 \end{aligned}
\end{equation}

The first two types of constraints are the \textit{marginal equations} of~\eqref{eq:the-BROJA-CP}. The third type of constraints connects the $p$-variables with the $q$-variables which will be denoted as \textit{coupling equations}. The generalized inequality connects these to the variables forming the objective function.
\begin{proposition}
	The exponential cone program~\eqref{eq:BROJA-ExpCone} is equivalent to the Convex Program~\eqref{eq:the-BROJA-CP}.	
\end{proposition}

\begin{proof}
	Let $\P_{\CP}(b)$ and $\P_{\exp}(b)$ be the feasible region of~\eqref{eq:the-BROJA-CP} and~\eqref{eq:BROJA-ExpCone} respectively. We define the following
	\begin{equation}
		\begin{aligned}
			f: \P_{\CP}(b)	&\to 	&&\P_{\exp}(b)\\
				q_{x,y,z} 	&\to	&& 
									f(q_{x,y,z}):=\begin{cases}
										(q_{x,y,z}\ln\frac{q_{*,y,z}}{q_{x,y,z}}, q_{*,y,z}, q_{x,y,z}) &\text{if $q_{xyz}>0$}\\
										(0,q_{*,y,z}, q_{x,y,z})										&\text{if $q_{x,y,z}=0$.}
									\end{cases}
		\end{aligned}
	\end{equation}
	For $q_{x,y,z}\in\P_{\CP}$, we have
	\begin{equation*}
		(-1,0,0)^T\cdot f(q_{x,y,z}) =	\begin{cases}
											q_{x,y,z}\ln\frac{q_{x,y,z}}{q_{*,y,z}} &\text{if $q_{xyz}>0$}\\
											0										&\text{if $q_{xyz}=0$}\\
										\end{cases}  
	\end{equation*} 
	and since conditional entropy at $q_{x,y,z}=0$ vanishes, then the objective function of~\eqref{eq:the-BROJA-CP} evaluated at $q\in\P_{\CP}$ is equal to that of~\eqref{eq:BROJA-ExpCone} evaluated at $f(q)$. If $(r,p,q)\in\P_{\exp}\backslash\Ima(f)$, then there exists $x,y,z$ such that $r_{x,y,z}<q_{x,y,z}\ln\frac{p_{x,y,z}}{q_{x,y,z}}$ and so
	$$-\sum_{x,y,z} r_{x,y,z} >\sum_{x,y,z} q_{x,y,z}\ln\frac{q_{x,y,z}}{p_{x,y,z}}.$$
\end{proof}
The dual problem of~\eqref{eq:BROJA-ExpCone} is 
\begin{subequations}
	\begin{align}
		\text{maximize }		& -\sum_{x,y}\lambda_{x,y}b^\y_{x,y} -\sum_{x,z}\lambda_{x,z}b^\z_{x,z} \notag\\
		\text{subject to}\quad	& 														- \nu_{x,y,z}^1	= 1	&&\text{ for all $(x,y,z)\in \mathbf{X}\times\mathbf{Y}\times\mathbf{Z}$}\label{eq:DExp-1}\\
		{}						&		\mu_{x,y,z} 									- \nu_{x,y,z}^2 = 0	&&\text{ for all $(x,y,z)\in \mathbf{X}\times\mathbf{Y}\times\mathbf{Z}$}\label{eq:DExp-2}\\
		{}						& 	-	\mu_{*,y,z}	-	\lambda_{x,y} -	\lambda_{x,z}	- \nu_{x,y,z}^3 = 0	&&\text{ for all $(x,y,z)\in \mathbf{X}\times\mathbf{Y}\times\mathbf{Z}$}\label{eq:DExp-3}\\
		{}						& (\nu_{x,y,z}^1,\nu_{x,y,z}^2,\nu_{x,y,z}^3)	\ge_{\mathcal K^*_{\exp}} 0	&&\text{ for all $(x,y,z)\in \mathbf{X}\times\mathbf{Y}\times\mathbf{Z}$}\label{eq:DExp-4}
	\end{align}
\end{subequations}	
Using the definition of $\mathcal K_{\exp}^*$ the system consisting of~\eqref{eq:DExp-1},~\eqref{eq:DExp-2},~\eqref{eq:DExp-3}, and~\eqref{eq:DExp-4} is equivalent to 
\begin{equation*}
		\lambda_{x,y} + \lambda_{x,z} + \mu_{*,y,z} + 1 +\ln(-\mu_{x,y,z}) \ge 0\mskip20mu\text{ for all $(x,y,z)\in\mathbf{X}\times\mathbf{Y}\times\mathbf{Z}$}
\end{equation*}
and so the dual problem of~\eqref{eq:BROJA-ExpCone} can be formulated as 
\begin{equation}\label{eq:BROJA-DExpCone}\tag{\mbox{D-EXP}}
	\begin{aligned}
		\text{maximize }		& -	\sum_{x,y}\lambda_{x,y}b^\y_{x,y} -	\sum_{x,z}\lambda_{x,z}b^\z_{x,z} \\
		\text{subject to}\quad
								& \lambda_{x,y}	+	\lambda_{x,z}	+	\mu_{*,y,z}	+ 1	+	\ln(-\mu_{x,y,z}) \ge 0	&&\text{ for all $(x,y,z)\in\mathbf{X}\times\mathbf{Y}\times\mathbf{Z}$}
	\end{aligned}
\end{equation}
\begin{proposition}\label{prop:int-point}
 Strong duality holds for the primal-dual pair~\eqref{eq:BROJA-ExpCone},~\eqref{eq:BROJA-DExpCone}.
\end{proposition}
\begin{proof} We assume that $b^\y_{x,y},b^\z_{x,z}>0.$ Consider the point $\tilde{t}$ with $\tilde{t}_{x,y,z} = (\tilde{r}_{x,y,z},\tilde{p}_{x,y,z},\tilde{q}_{x,y,z})$ such that  
	\begin{equation}
	\begin{aligned}
		\tilde{r}_{x,y,z} &:= \tilde{q}_{x,y,z}\log\frac{\tilde{p}_{x,y,z}}{\tilde{q}_{x,y,z}} - 100\\
		\tilde{p}_{x,y,z} &:= \tilde{q}_{*,y,z} \\
		\tilde{q}_{x,y,z} &:= \frac{b^\y_{x,y}\cdot b^\z_{x,z}}{b^\y_{x,*}}.
	\end{aligned}
	\end{equation}
	$\tilde{t}$ is an interior point of~\eqref{eq:BROJA-ExpCone}, we refer to~\cite{makkeh:phd:2018} for the proof. Hence by Theorem~\ref{thm:dual}, strong duality holds for the primal-dual pair~\eqref{eq:BROJA-ExpCone},~\eqref{eq:BROJA-DExpCone}. 

\end{proof}
 Weak and strong duality in their turn will provide the user a measure for the quality of the returned solution, for more details see Section~\ref{subsec:output}.

%% file: implementation.tex
We implemented the exponential cone program~\eqref{eq:BROJA-ExpCone} into Python and used a conic optimization solver to get the desired solution. Note that we are aware of only two conic optimization software toolboxes which allow to solve Exponential Cone Programs, ECOS and SCS. The current version of \textsc{Broja\_2pid} utilities ECOS to solve the Exponential Cone Program~\eqref{eq:BROJA-ExpCone}. ECOS\footnote{We use the version from Nov~8, 2016.} is a lightweight numerical software for solving Convex Cone programs~\cite{Domahidi-Chu-Boyd:ECOS:13}. 

This section describes the \textsc{Broja\_2pid} package form the user's perspective.  We briefly explain how to install \textsc{Broja\_2pid}. Then we illustrate the framework of \textsc{Broja\_2pid} and its functions. Further, we describe the input, tuning parameters, and output.
\subsection{Installation}
To install \textsc{Broja\_2pid} you need Python to be installed on your machine. Currently you need to install ECOS, the Exponential Cone solver. To do that, you most likely\shellinline{pip3 install ecos}. In case of having troubles installing ECOS we refer to its Github repository~\url{https://github.com/embotech/ecos-python}. Finally you need to\shellinline{gitclone} the Github link of~\textsc{Broja\_2pid} and it is ready to be used.
\subsection{Computing Bivariate PID}
	In this subsection, we will explain how \textsc{Broja\_2pid} works. In Figure~\ref{fig:and-gate}, we present a script as an example of using \textsc{Broja\_2pid} package to compute the partial information decomposition of the \textsc{And} distribution, $X = Y\operatorname{AND}Z$ where $Y$ and $Z$ are independent and uniformly distributed in $\{0,1\}$.
	\begin{figure}[H]
		\centering
		\scalebox{0.7}[0.7]{
			\pythonexternal{figures/test_and_gate.py}
			}
		\caption{Computing the partial information decomposition of the \textsc{And} gate using \textsc{Broja\_2pid}.} \label{fig:and-gate}
	\end{figure} 
	
	We will go through the example (Figure~\ref{fig:and-gate}) to explain how \textsc{Broja\_2pid} works. The main function in \textsc{Broja\_2pid} package is\pythoninline{pid()}. It is a wrap up function which is used to compute the partial information decomposition. First,\pythoninline{pid()} prepares the ``ingredients'' of~\eqref{eq:BROJA-ExpCone}. Then it calls the Cone Programming solver to find the optimal solution of~\eqref{eq:BROJA-ExpCone}. Finally, it receives from the Cone Programming solver the required solution to compute the decomposition. 
	
	The ``ingredients'' of~\eqref{eq:BROJA-ExpCone} are the marginal and coupling equations, generalized inequalities, and the objective function. So,\pythoninline{pid()} needs to compute and store $b^\y_{x,y}$ and $b^\z_{x,z}$, the marginal distributions of $(X,Y)$ and $(X,Z)$. For this,\pythoninline{pid()} requires a distribution of $X,Y,$ and $Z$. In Figure~\ref{fig:and-gate}, the distribution comes from the \textsc{And} gate where $X = Y\operatorname{AND}Z.$ 
	
	The distribution must always be defined as a mapping (Python dictionary) of $(x,y,z)$, \textit{triplet}, to its respective probability, \textit{number}. E.g., the triplet $(0,0,0)$ occurs with probability $\nicefrac{1}{4}$ and so on for the other triplets. So \textsc{And} distribution is defined as a Python dictionary,\pythoninline{andgate=dict()} where\pythoninline{andgate[ (0,0,0) ]=.25} is assigning the key ``$(0,0,0)$'' a value ``0.25'' and so on.
	
	Note that the user does not have to add the triplets with zero probability to the dictionary since\pythoninline{pid()} will always discard such triplets. In~\cite{makkeh2017bivariate}, the authors discussed in details how to handle the triplets with zero probability. The input of\pythoninline{pid()} is explained in details in the following subsection.
	
	Now we briefly describe how\pythoninline{pid()} proceeds to return the promised decomposition. \pythoninline{pid()} calls the Cone Programming solver and provides it with the ``ingredients'' of~\eqref{eq:BROJA-ExpCone} as a part of the solver's input. The solver finds the optimal solution of~\eqref{eq:BROJA-ExpCone} and~\eqref{eq:BROJA-DExpCone}. When the solver halts it returns the primal and dual solutions. Using the returned solutions,\pythoninline{pid()} computes the decomposition based on equation~\eqref{eq:broja-short-def}.  The full process is explained in Figure~\ref{fig:flow-chart}.
	
	Finally,\pythoninline{pid()} returns a Python dictionary,\pythoninline{returndata} containing the partial information decomposition and information about the quality of the Cone Programming solver's solution. In Subsection~\ref{subsec:output} we give a detailed explanation on how to compute the quality's data and Table~\ref{tab:output} contains a description of the keys and values of\pythoninline{returndata}.

	E.g., in the returned dictionary\pythoninline{returndata} for the \textsc{And} gate,\pythoninline{returndata['CI']} contains the quantity of synergistic information and\pythoninline{returndata['Num_err'][0]} the maximum primal feasibility violation of~\eqref{eq:BROJA-ExpCone}. 
	
	Note that conic optimization solver is always supposed to return a solution. So, \textsc{Broja\_2pid} will raise an exception,\pythoninline{BROJA\_2PID\_Exception}, when no solution is returned. 
%
%
%

%
	\begin{figure}
		\centering
		\includegraphics[height=8cm, width=15cm]{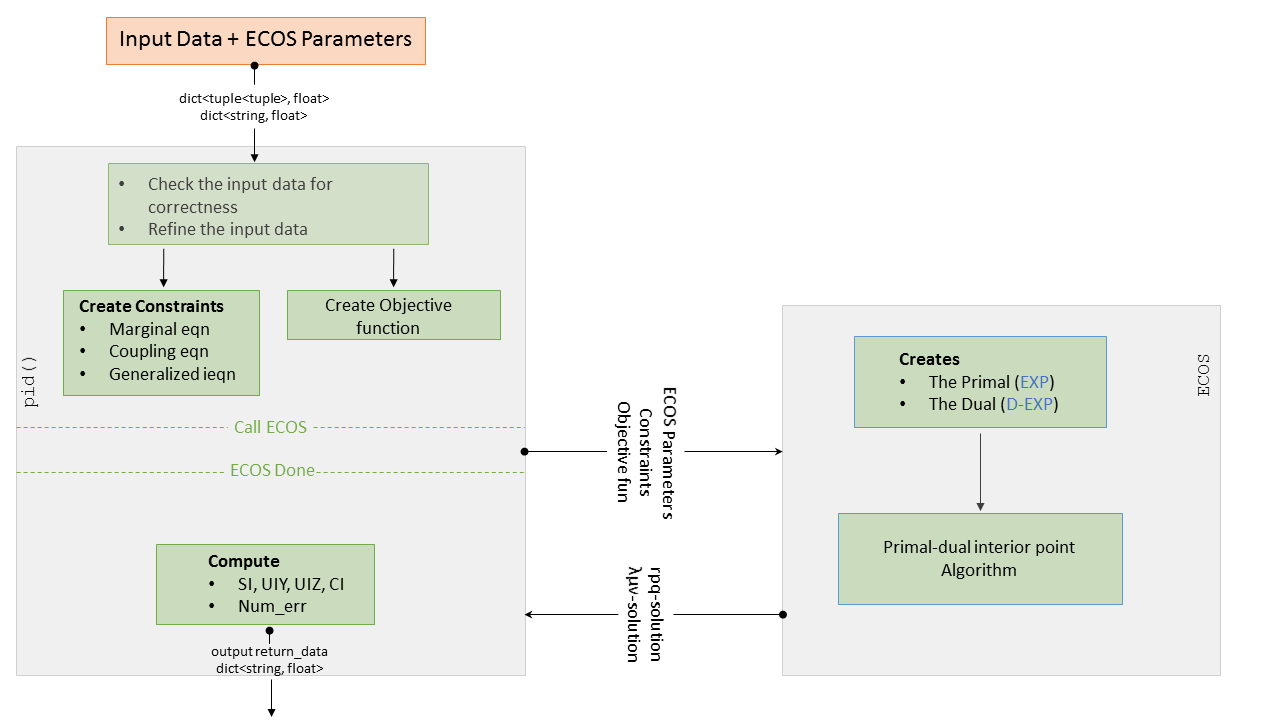}
		\caption{\textsc{Broja\_2pid} workflow. Left is the flow in \texttt{pid()}. Right is the flow in ECOS. The arrows with oval tail indicate passing of data whereas the ones with line tail indicate time flow.}
		\label{fig:flow-chart}
	\end{figure}     
\subsection{Input and Parameters}\label{subsec:input}
	In \textsc{Broja\_2pid} package,\pythoninline{pid()} is the function which the user needs to compute the partial information decomposition. The function\pythoninline{pid()} takes as input a Python dictionary.
	
	The Python dictionary will represent a probability distribution. This distribution computes the vectors $b^\y_{x,y}$ and $b^\z_{x,z}$ for the  the marginal equations in \eqref{eq:BROJA-ExpCone}. A key of the Python dictionary is a \textit{triplet} of $(x,y,z)$ which is a possible outcome of the random variables $X,Y,$ and~$Z$. A value of the key~$(x,y,z)$ in the Python dictionary is a \textit{number} which is the probability of $X=x,Y=y,$ and $Z=z$. 
	
	\paragraph{Solver parameters.} The Cone Programming solver has to make sure while seeking the optimal solution of~\eqref{eq:BROJA-ExpCone} that $w$ and $(\eta, \theta)$ are feasible and (ideally) should halt when the duality gap is zero, i.e., $w$ and $(\eta, \theta)$ are optimal. But  $w$ and $(\eta, \theta)$ entries belong to $\RR$ and computers represent real numbers up to floating precision. So the Cone Programming solver consider a solution feasible, none of the constraints are violated, or optimal, duality gap is zero, up to a numerical precision (tolerance). The Cone Programming solver allows the user to modify the feasibility and optimality tolerances along with couple other parameters which are described in Table~\ref{tab:ecos}. 
	\begin{table}[H]
		\centering
		\scalebox{0.9}[0.9]{
		\begin{tabular}{llc}
		\toprule
		Parameter 				& Description 												& Recommended Value \\
		\midrule
		\texttt{feastol}		& 	primal/dual feasibility tolerance  						& $10^{-7}$         \\
		\texttt{abstol}			&	absolute tolerance on duality gap						& $10^{-6}$         \\
		\texttt{reltol}			&	relative tolerance on duality gap						& $10^{-6}$         \\
		\texttt{feastol\_inacc} &	primal/dual infeasibility \textit{relaxed} tolerance	& $10^{-3}$         \\
		\texttt{abstol\_inacc}	&	absolute \textit{relaxed} tolerance on duality gap		& $10^{-4}$         \\
		\texttt{reltol\_inacc}	&	\textit{relaxed} relative duality gap					& $10^{-4}$         \\	
		\texttt{max\_iter}		&	maximum number of iterations that ``ECOS'' does			& $100$             \\ 
		\bottomrule
		\end{tabular}%
		}%
		\caption[Parameters (tolerances) of ECOS]{Parameters (tolerances) of ECOS\footnotemark.}\label{tab:ecos}
	\end{table}
	\footnotetext{The parameters \texttt{reltol} is not recommended to be set higher. For more explanation see~\url{https://github.com/embotech/ecos}.}
	
		\begin{wrapfigure}[8]{R}{0.5\textwidth}
			\centering
		  	\pythonexternal{figures/test_parameters.py}
		    \caption{Tuning parameters} \label{fig:tune}
		\end{wrapfigure}
	In order to change the default Cone Programming solver parameters, the user should pass them to\pythoninline{pid()} as a dictionary. E.g.~in Figure~\ref{fig:tune}, we change the maximum number of iterations which the solver can do. For this we created a dictionary,\pythoninline{parms=dict()}. Then we set a desired value,\pythoninline{1000}, for the key\pythoninline{'max_iter'}. Finally, we are required to pass\pythoninline{parms} to\pythoninline{pid()} as a dictionary,\pythoninline{pid(andgate,**parms)}. Note that in the defined dictionary\pythoninline{parms}, the user only needs to define the keys for which the user wants to change the values.
	
	There is another type of parameters, namely,\pythoninline{output} which is an integer in~$\{0,1,2\}$. It is a parameter which determines the printing mode of\pythoninline{pid()}. Meaning that it allows the user to control what will be printed on the screen. Table~\ref{tab:printing-mode} gives a detailed description of the printing mode.
		\begin{table}[H]
			\centering
			\scalebox{0.9}[0.9]{
			\begin{tabular}{ll}
				\toprule
				\texttt{output}	& Description 																	\\
				\midrule
				0 (default)		& \pythoninline{pid()} prints its output (python dictionary, see Subsection~\ref{subsec:output}).						\\
				1				& In addition to\pythoninline{output= 0}, \pythoninline{pid()} prints a flags when it starts preparing~\eqref{eq:BROJA-ExpCone}\\
			     				& and another flag when it calls the conic optimization solver.															\\
				2				& In addition to\pythoninline{output= 1}, \pythoninline{pid()} prints the conic optimization solver's output\footnotemark.	\\
				\bottomrule
			\end{tabular}%
			}%
			\caption{Description of the printing mode in \texttt{pid()}.}\label{tab:printing-mode}
		\end{table}
		\footnotetext{The conic optimization solver usually prints out the problem statistics and the status of optimization.}
    Currently we are only using ECOS to solve the Exponential Cone Program but in the future we are going to add the SCS solver. For the latter reason, the user should determine which solver will be used in the computations. E.g., setting \pythoninline{cone_solver="ECOS"} will utilize ECOS in the computations. 
\subsection{Returned Data}\label{subsec:output}
	The function\pythoninline{pid()} returns a Python dictionary called\pythoninline{returndata}. Table~\ref{tab:output} describes the returned dictionary.
	\begin{table}[H]
		\centering
		\scalebox{1}[1]{
		\begin{tabular}{ll}
		\toprule
		Key 							& Value 																	\\
		\midrule
		\pythoninline{'SI'}				& Shared information, $\SI(X;Y,Z)$.\footnotemark						    \\
		\pythoninline{'UIY'}			& Unique information of $Y$, $\UI(X;Y\backslash Z)$. 						\\
		\pythoninline{'UIZ'}			& Unique information of $Z$, $\UI(X;Z\backslash Y)$.						\\
		\pythoninline{'CI'} 			& Synergistic information, $\CI(X;Y,Z)$.									\\
		\pythoninline{'Num_err'}		& information about the quality of the solution. 							\\
		\pythoninline{'Solver'}			& name of the solver used to optimize~\eqref{eq:the-BROJA-CP}\footnotemark.	\\ 
		\bottomrule
		\end{tabular}%
		} %
		\caption{Description of {\ttm returndata}, the Python dictionary returned by \texttt{pid()}.}\label{tab:output}
	\end{table}
	\footnotetext[4]{All information quantities are returned in bits.}
	\footnotetext{In this version we are only using ECOS, but other solvers might be added in the future.}
	Let $w,\eta,$ and $\theta$ be the lists returned by the Cone Programming solver where $w_{x,y,z} = [r_{x,y,z},p_{x,y,z},q_{x,y,z}],$ $\eta_{x,y,z}= [\lambda_{x,y},\lambda_{x,z}, \mu_{x,y,z}],$ and $\theta_{x,y,z} = [\nu_{x,y,z}]$. Note that $w$ is the primal solution and $(\eta,\theta)$ is the dual solution. The dictionary\pythoninline{returndata} gives the user access to the partial information decomposition, namely, shared, unique, and synergistic information. The partial information decomposition is computed using only the positive values of $q_{x,y,z}$. The value of the key\pythoninline{'Num_err'} is a triplet such that the primal feasibility violation is\pythoninline{returndata['Num_err'][0]}, the dual feasibility violation is \pythoninline{returndata['Num_err'][1]}, and\pythoninline{returndata['Num_err'][2]} is the duality gap violation. In what follows, we will explain how we compute the violations of primal and dual feasibility in addition to that of duality gap. 
	
	The primal feasibility of~\eqref{eq:BROJA-ExpCone} is
	\begin{equation}\label{eq:primal-feas}
		\begin{split}
			q_{x,y,*} = b^\y_{x,y}\\
			q_{x,*,z} = b^\z_{x,z}\\
			q_{*,y,z} = p_{x,y,z}\\
			(-r_{x,y,z},-p_{x,y,z},-q_{x,y,z}) &\le_{\mathcal K_{\exp}} 0
		\end{split}
	\end{equation}
	We check the violation of $q_{x,y,z}\ge 0$ which is required by $\mathcal K_{\exp}$. Since all the non-positive $q_{x,y,z}$ are discarded when computing the decomposition, we check if the marginal equations are violated using only the positive $q_{x,y,z}$. The coupling equations are ignored since they are just assigning values to the $p_{x,y,z}$ variables. So,\pythoninline{returndata['Num_err'][0]} (primal feasibility violation) is computed as follows, 
	\begin{equation*}
		\begin{split}
			q'_{x,y,z} &=\begin{cases}
				0 &~\text{if $q_{x,y,z} \le 0$}\\
				q_{x,y,z} &~\text{otherwise}
			\end{cases}\\
			\pythoninline{returndata['Num_err'][0]} &=\max_{x,y,z}( \abs{q'_{x,y,*} - b^\y_{x,y}}, \abs{q'_{x,*,z} - b^\z_{x,z}}, -q_{x,y,z})
		\end{split}
	\end{equation*} 
	The dual feasibility of~\eqref{eq:BROJA-DExpCone} is
	\begin{equation}\label{eq:daul-feas}
		\lambda_{x,y}	+	\lambda_{x,z}	+	\mu_{*,y,z}	+ 1	+	\ln(-\mu_{x,y,z}) \ge 0
	\end{equation}
	For dual feasibility violation, we check the non-negativity of Equation~\eqref{eq:daul-feas}. So, the error\pythoninline{returndata['Num_err'][1]} is equal to 
	\begin{equation*}
		\min_{x,y,z}(\lambda_{x,y}+\lambda_{x,z}+\mu_{*,y,z}+1+\ln(-\mu_{x,y,z}), 0)
	\end{equation*}
	When $w$ is the optimal solution of~\eqref{eq:BROJA-ExpCone}, we have 
	\begin{equation*}
	-\sum_{x,y,z} r_{x,y,z} = \sum_{x,y,z} q_{x,y,z}\log\frac{q_{x,y,z}}{q_{*,y,z}} = -H(X\mid Y,Z).
	\end{equation*}
	The duality gap of~\eqref{eq:BROJA-ExpCone} and~\eqref{eq:BROJA-DExpCone} is 
	\begin{equation}\label{eq:d-gap}
		-H(X\mid Y,Z) + \lambda^Tb,
	\end{equation}
	where%
	\begin{equation*}	 
			\lambda^Tb =\sum_{x,y}\lambda_{x,y}b^\y_{x,y} +\sum_{x,z}\lambda_{x,z}b^\z_{x,z}.
	\end{equation*}
	Since weak duality implies $H(X\mid Y,Z)\le \lambda^Tb$, so we are left to check the non negativity of~\eqref{eq:d-gap} to inspect the duality gap. So,\pythoninline{returndata['Num_err'][2]} is given by,%
	\begin{equation*}
			\max(-H(X\mid Y,Z) + \lambda^Tb,0)
	\end{equation*}

%% file: case_study.tex
In this section, we will test the performance of \textsc{Broja\_2pid} on three types of instances. We will describe each type of instances and show the results of testing \textsc{Broja\_2pid} against each one of them. The first two types are used as primitive validation tests. However, the last type is used to evaluate the accuracy and efficiency of \textsc{Broja\_2pid} in computing the partial information decomposition. We used a computer server with Intel(R) Core(TM) i7-4790K CPU (4 cores) and 16GB of RAM to solve the instances. All computations were done using one core. In all our instances, the vectors $b^\y$ and $b^\z$ are marginals computed from an input probability distribution $p$ on $\mathbf{X}\times\mathbf{Y}\times\mathbf{Z}.$

\subsection{Paradigmatic Gates}
The following set of instances have been studied extensively throughout the literature. The partial information decomposition of the set of instances is known~\cite{griffith2014quantifying}. Despite their simplicity, they acquire desired properties of shared or synergistic quantities. 
\subsubsection{Data}
The first type of instances is based on the ``gates'' (\textsc{Rdn}, \textsc{Unq}, \textsc{Xor}, \textsc{And}, \textsc{RdnXor}, \textsc{RdnUnqXor}, \textsc{XorAnd}) described in Table~1 of~\cite{bertschinger-rauh-olbrich-jost-ay:quantify:2014}.  Each gate is given as a function $(x,y,z)=\texttt G(W)$ which maps a (random) input~$W$ to a triple $(x,y,z)$. The inputs are sampled uniformly at random, whereas, in Table~1 of~\cite{bertschinger-rauh-olbrich-jost-ay:quantify:2014} the inputs are independent and identically distributed.

\subsubsection{Testing}
All the gates are implemented as dictionaries and\pythoninline{pid()} is called successively with different printing modes to compute them. The latter is coded into the script file at the Github directory \shellinline{Testing/test_gates.py}. 
The values of the partial information decomposition for all the gates distributions (when computed by\pythoninline{pid()}) were equal to the actual values up to precision error of order $10^{-9}$ and the computations were done in less than a millisecond.

\subsection{\textsc{Copy} Gate}
The \textsc{Copy} gate requires a large number of variables and constraints-- see below for details. So, we used it to test the memory efficiency of the \textsc{Broja\_2pid} solver. Since its decomposition is known, it also provides to some extent a validation for the correctness of the solution in large systems. 
\subsubsection{Data}
\textsc{Copy} gate is the mapping of $(y,z)$ chosen uniformly at random to a triplet $(x,y,z)$ where $x = (y,z)$. The \textsc{Copy} distribution overall size scales as $|\mathbf{Y}|^2\times |\mathbf{Z}|^2$ where $y,z\in\mathbf{Y}\times\mathbf{Z}$. Proposition 18 in~\cite{bertschinger-rauh-olbrich-jost-ay:quantify:2014} shows that the partial information decomposition of~\textsc{Copy} gate is 
\begin{subequations}
	\begin{align}
		\CI(X;Y,Z)				&= 0\notag\\
		\SI(X;Y,Z)				&= \MI(Y;Z)\notag\\
		\UI(X;Y\backslash Z)	&= H(Y\mid Z)\notag\\
		\UI(X;Z\backslash Y)	&= H(Z\mid Y)\notag
	\end{align}
\end{subequations}
Since $Y$ and $Z$ are independent random variables, then $\UI(X;Y\backslash Z) = H(Y)$ and $\UI(X;Z\backslash Y)= H(Z)$ and $SI(Y;Z) = 0$.
\subsubsection{Testing} %
The \textsc{Copy} distributions is generated for different sizes of $\mathbf{Y}$ and $\mathbf{Z}$ where $\mathbf{Y}=[m]$ and $\mathbf{Z}=[n]$ for $m,n\in\NN\backslash\{0\}$. Then \pythoninline{pid()} is called to compute the partial information decomposition for each pair of $m,n$. Finally, the\pythoninline{returndata} dictionary is printed along with the running time of the \textsc{Broja\_2pid} solver and the deviations of\pythoninline{returndata['UIY']} and\pythoninline{returndata['UIZ']} from $H(Y)$ and $H(Z)$ respectively. The latter process is implemented in\shellinline{Testing/test_large_copy.py}. The worst deviation was of percentage at most $10^{-8}$ for any $m,n\le 100.$

\subsection{Random Probability Distributions}
This is the main set of instances for which we test the efficiency of \textsc{Broja\_2pid} solver. It has three subsets of instance where each one of them is useful for an aspect of efficiency when the solver is used against large systems. This set of instances had many hard distributions in the sense that the feasible region of~\eqref{eq:broja-short-def} is ill or its solution lies on the boundary.
\subsubsection{Data}
The last type of instances are joint distributions of $(X,Y,Z)$ sampled uniformly at random over the probability simplex. We have three different sets of the joint distributions depending on the size of $\mathbf{X},\mathbf{Y},$ and $\mathbf{Z}$. 
\begin{enumerate}[label=\alph{enumi})]
	\item For set~1, we fix $|\mathbf{X}|=|\mathbf{Y}|=2$ and vary $|\mathbf{Z}|$ in $\{2,3,\dots,14\}$. Then, for each size of $Z$, we sample uniformly at random 500 joint distribution of $(X,Y,Z)$ over the probability simplex.
	\item For set~2, we fix $|\mathbf{X}|=|\mathbf{Z}|=2$ and vary $|\mathbf{Y}|$ in $\{2,3,\dots,14\}$. Then, for each value of $|Y|$, we sample uniformly at random 500 joint distribution of $(X,Y,Z)$ over the probability simplex.
	\item For set~3, we fix $|\mathbf{X}|=|\mathbf{Y}|=|\mathbf{Z}|= s$ where $s\in\{8,9,\dots,18\}$. Then, for each $s$, we sample uniformly at random 500 joint distribution of $(X,Y,Z)$ over the probability simplex.  
\end{enumerate}
Note that in each set, instances are grouped according to the varying value, i.e., $|\mathbf{Y}|,|\mathbf{Z}|,$ and $s$ respectively.%

\subsubsection{Testing}
The instances were generated using the Python script\shellinline{Testing/test_large_randompdf.py}. The latter script takes as command-line arguments $|\mathbf{X}|,|\mathbf{Y}|,|\mathbf{Z}|$ and the number of joint distributions of $(X,Y,Z)$ the user wants to sample from the probability simplex. E.g. if the user wants to create the instance of set~1 with $|\mathbf{Z}|= 7$ then the corresponding command-line is\shellinline{python3 test_large_randompdf.py 2 2 7 500}. The script outputs the\pythoninline{returndata} along with the running time of \textsc{Broja\_2pid} solver for each distribution and finally it prints the empirical average over all the distributions of $SI(X;Y,Z), UI(X;Y\backslash Z), UI(X;Y\backslash Z),$ $CI(X;Y,Z),$ and of the running time of \textsc{Broja\_2pid} solver.

In what follows for each of the sets, we look at $UI(X;Y\backslash Z)$ to validate the solution, the\pythoninline{returndata['Num_err']} triplet to examine the quality of the solution, and the running time to analyze the efficiency of the solver. 

\paragraph{Validation.} Sets~1 and~2 are mainly used to validate the solution of \textsc{Broja\_2pid}. For set~1, when $|\mathbf{Z}|$ is considerably larger than $|\mathbf{Y}|$, the amount of unique information that $Y$ has about $X$ is more likely to be small for any sampled joint distribution. So for set~1, the average $\UI(X;Y\backslash Z)$ is expected to decrease as the size of $Z$ increases. Whereas for set~2, $\UI(X;Y\backslash Z)$ is expected to increase as the size of $Y$ increases, i.e., when $|\mathbf{Y}|$ is considerably larger that $|\mathbf{Z}|$. \textsc{Broja\_2pid} shows such behavior of $\UI(X;Y\backslash Z)$ on the instances of  sets~1 and~2 see Figures~\ref{fig:UIY}.
\begin{figure}[H]
	\centering
	\subfigure[$\UI(X;Y\backslash Z)$ of set~1]{\includegraphics[height=4.96cm]{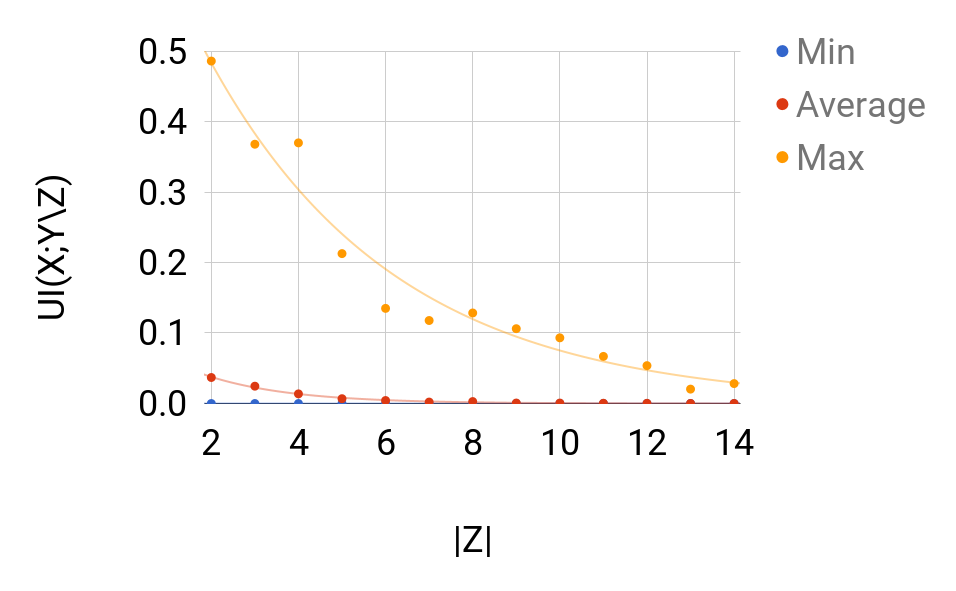}}
	\subfigure[$\UI(X;Y\backslash Z)$ of set~2]{\includegraphics[height=4.96cm]{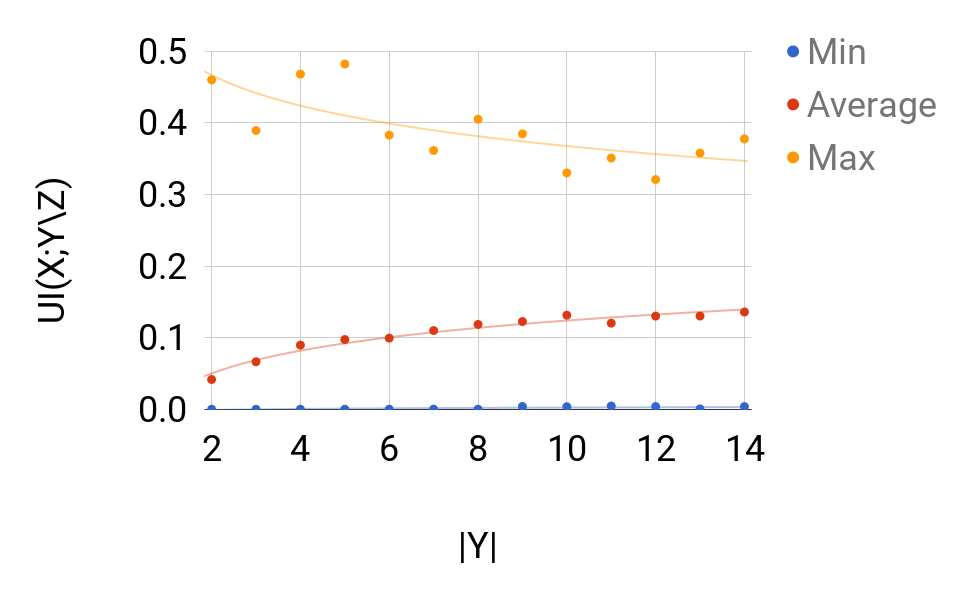}}
	\caption{For each group of instances in sets~1 and 2: (a) and (b) show the instance with the largest $\UI(X;Y\backslash Z)$, the average value of $\UI(X;Y\backslash Z)$ for the instances, and the instance with the smallest $\UI(X;Y\backslash Z)$.}
	\label{fig:UIY}
\end{figure}

\paragraph{Quality.} The solver did well on most of the instances. The percentage of solved instances to optimality was at least $99\%$ for each size in any set of instances. In Figures~\ref{fig:err}, we plot the successfully solved instances against the maximum value of the numerical error triplet\pythoninline{returndata['Num_err']}. On one hand, these plots show that whenever an instance is solved successfully the quality of the solution is good. On the other hand, when the Cone Programming solver fail to find an optimal solution for an instance, i.e., the primal feasibility or dual feasiblity or the duality gap is violated. We noticed that the duality gap,\pythoninline{returndata['Num_err'][2]}, was very large. Thus, these results reflect the reliability of the solution returned by \textsc{Broja\_2pid}. We may note that even when \textsc{Broja\_2pid} fails to solve an instance to optimality, it will return a solution\footnote{\textsc{Broja\_2pid} raise an expectation if the conic optimization solver fails to return a solution.}.

\begin{figure}[H]
	\centering
	\subfigure[Maximum numerical error of set~1]{\includegraphics[height=4.96cm]{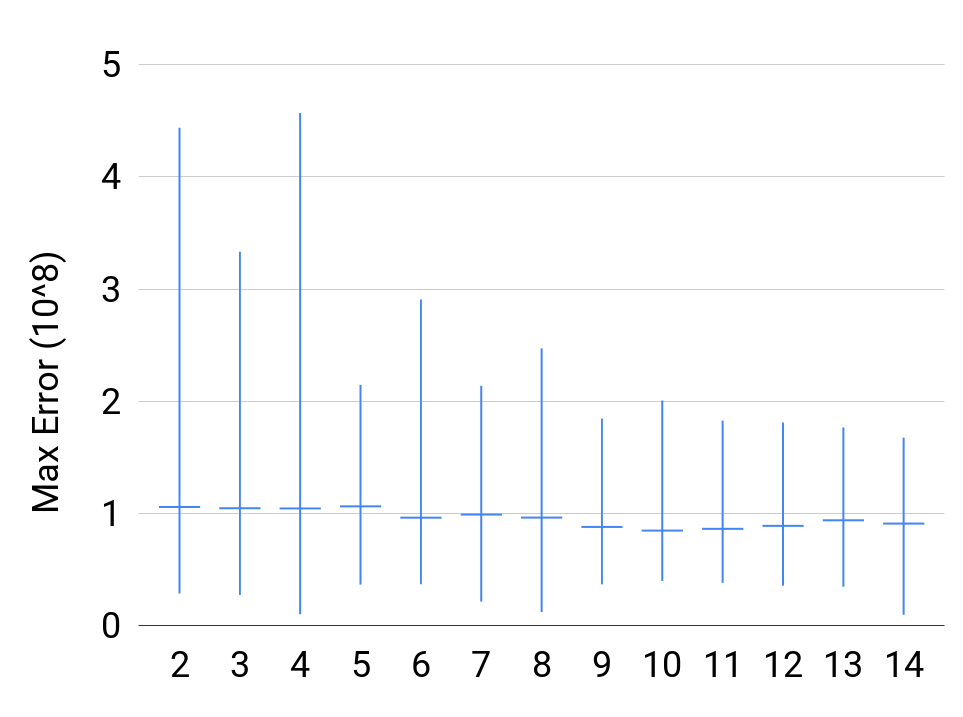}}
	\subfigure[Maximum~numerical error of set~2]{\includegraphics[height=4.96cm]{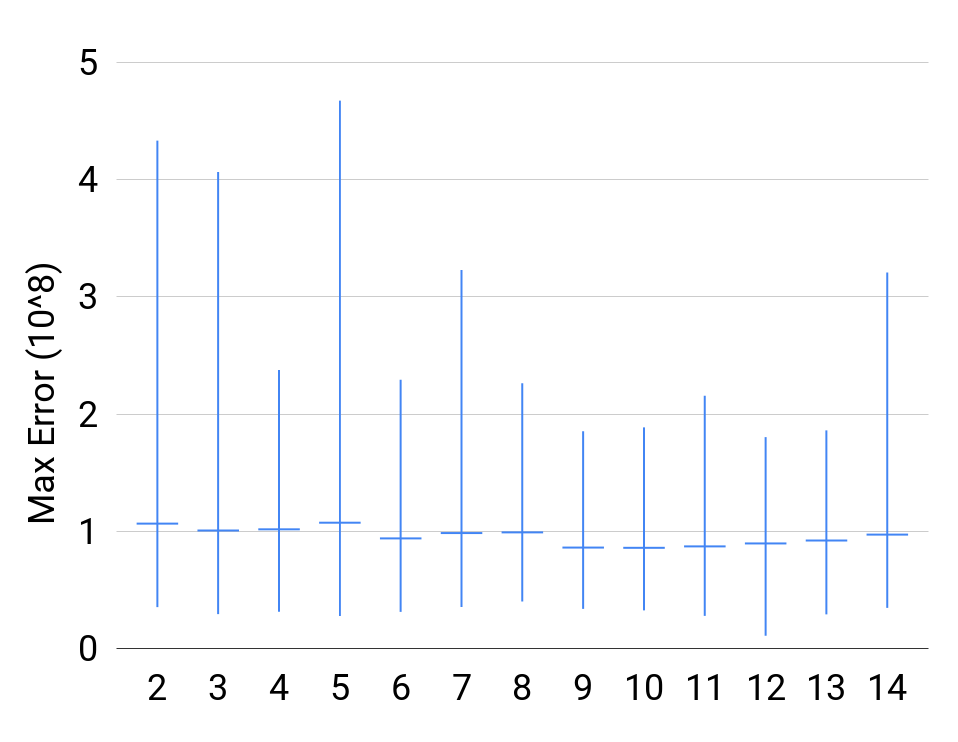}}
	\subfigure[Maximum~numerical error of set~3]{\includegraphics[height=4.96cm]{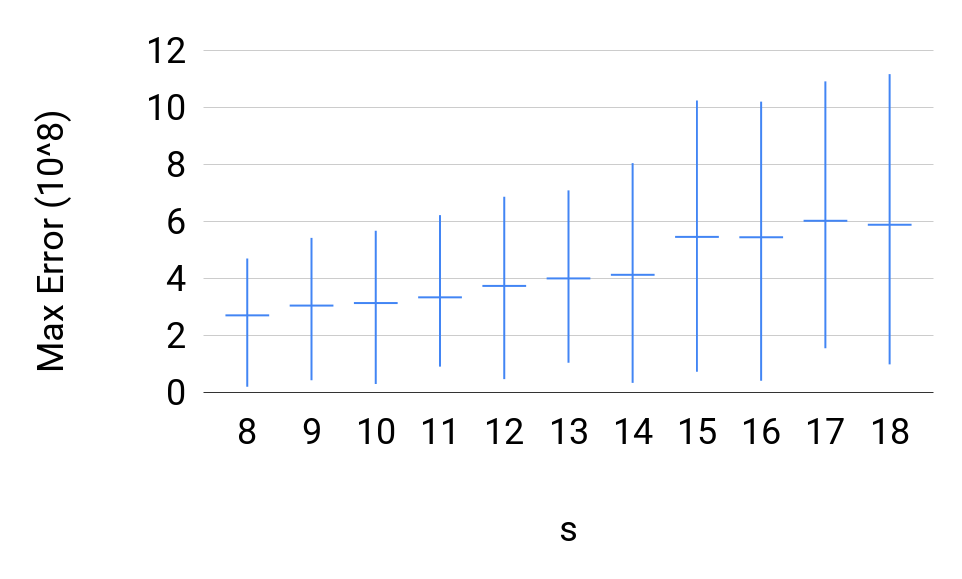}}
	\caption[font=small]{For each group of instances in sets~1,2, and 3: (a), (b), and (c) show the instance with the largest $\epsilon$, the average value of $\epsilon$ for the instances, and the instance with the smallest $\epsilon$; where $\epsilon$ is the maximum numerical error.}
	\label{fig:err}
\end{figure}

\paragraph{Efficiency.} In order to test the efficiency of \textsc{Broja\_2pid} in the sense of running time, we looked at set~3. The reason is that set~1 and~2 are small scale systems. Whereas, set~3 have a large input size mimicking large scale systems. Testing set~3 instances also reveals how the solver empirically scales with the size of input. Figure~\ref{fig:time} shows that the running time for \textsc{Broja\_2pid} solver against large instances was below 50~minutes. Furthermore, the solver has a scaling of $|\mathbf{X}|\times|\mathbf{Y}|\times|\mathbf{Z}|$, so on set~3, it scales as $N^3$ where $N$ is the size of input for the sampled distributions such that $|\mathbf{X}|=|\mathbf{Y}|=|\mathbf{Z}| = N$.
\begin{figure}[H]
	\centering
	\subfigure[$t^{1/6}$ versus $s$]{\includegraphics[height=4.96cm]{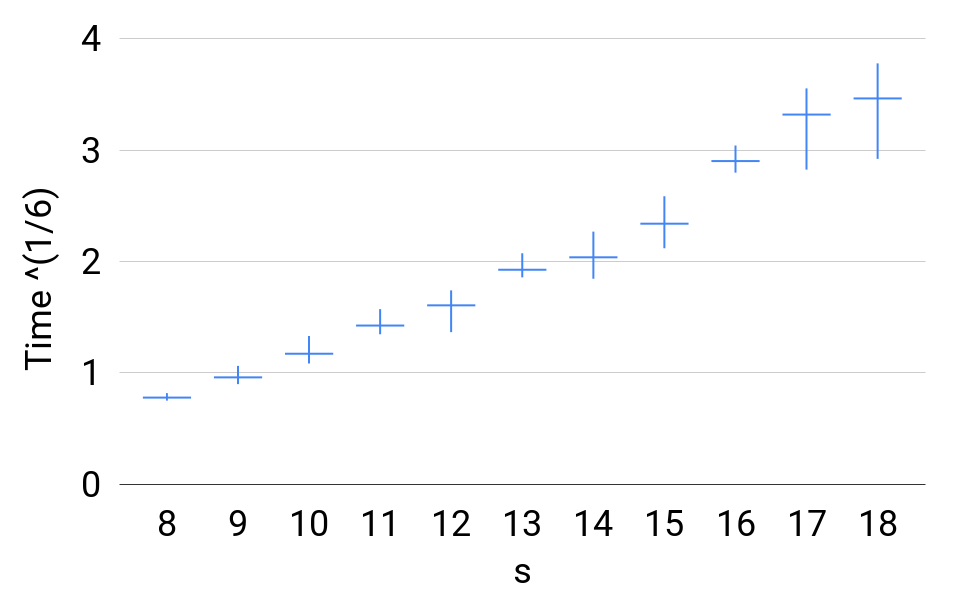}}
	\subfigure[$t^{1/3}$ versus $s$]{\includegraphics[height=4.96cm]{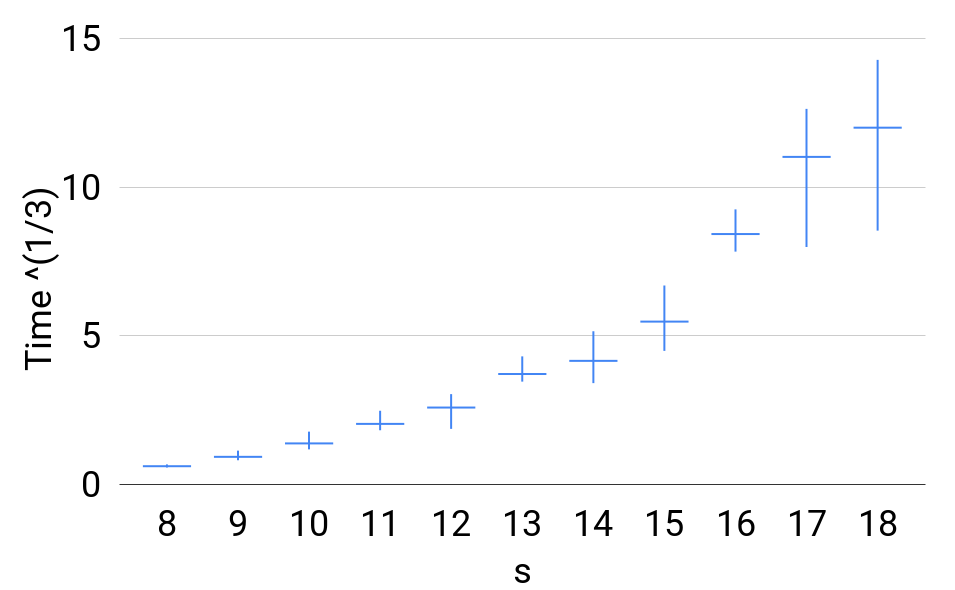}}
	\subfigure[$\nfrac{t}{10^3}$ versus $s$]{\includegraphics[height=4.96cm]{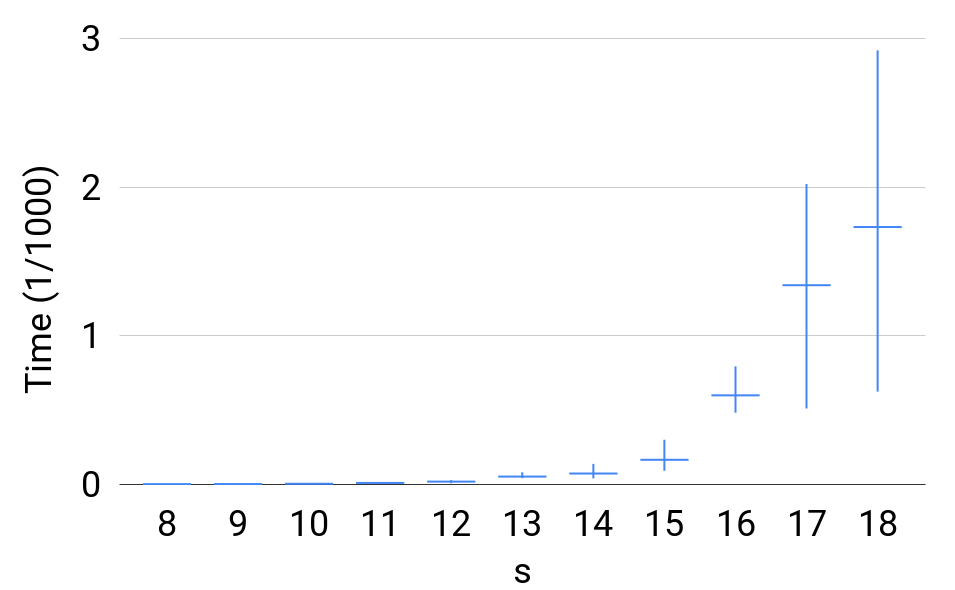}}
	\caption{For each group of instances in set~3: (a), (b), and (c) show the slowest instance, the average value of  running times, and the fastest instance; where the running time of \textsc{Broja\_2pid}, $t$ (secs), is scaled to $t^{1/6},t^{1/3},$ and $\nfrac{t}{10^3}$ respectively.}
	\label{fig:time}
\end{figure}

%% file: conclusion.tex
We are aware of one other Cone Programming solver with support for the Exponential Cone, SCS~\cite{ODonoghue-Chu-Parikh-Boyd:SCS:16}.  We are currently working on adding the functionality to our software.  When that is completed, giving the parameter \verb+cone_solver="SCS"+ to the function \verb+pid()+ will make our software use the SCS-based model instead of the ECOS-based one. (The models themselves are in fact different: SCS requires us to start from the dual exponential cone program~\eqref{eq:BROJA-DExpCone}.
SCS employs parallelized first-order methods which can be run on GPUs, so we expect a considerable speedup for large-scale problem instances.

We may note that other information theoretical functions can also be fitted into the exponential cone. Thus, with some modification, the model can be used to solve other problems. 


\subsection*{Thanks}
The authors would like to thank Patricia Wollstadt and Michael Wibral for their feedback on pre-production versions of our software.